\documentclass[reqno,oneside]{amsart}

\usepackage[danish,english]{babel} 
\usepackage[latin1]{inputenc} 
\usepackage[T1]{fontenc} 
\usepackage{lmodern}
\usepackage{amsmath,amsthm,amssymb,amsfonts,mathrsfs,latexsym} 
\usepackage{graphicx} 
\usepackage{verbatim} 
\usepackage[all]{xy} 
\usepackage[colorlinks, linkcolor=blue, citecolor=blue, urlcolor=blue, hypertexnames=true]{hyperref}  
\usepackage{multirow}

\newtheorem{thm}{Theorem}
\newtheorem{lem}[thm]{Lemma}
\newtheorem{prop}[thm]{Proposition}

\theoremstyle{definition}

\newtheorem{rem}[thm]{Remark}

\newcommand{\mc}[1]{\mathcal{#1}}

\newcommand{\mf}[1]{\mathfrak{#1}}
\newcommand{\mr}[1]{\mathrm{#1}}

\newcommand{\C}{\mathbb{C}}

\newcommand{\R}{\mathbb{R}}
\newcommand{\Z}{\mathbb{Z}}
\newcommand{\la}{\langle}
\newcommand{\ra}{\rangle}
\renewcommand{\epsilon}{\varepsilon}
\renewcommand{\phi}{\varphi}
\renewcommand{\tilde}{\widetilde}

\newcommand{\acts}{\curvearrowright}
\DeclareMathOperator{\SL}{SL}
\DeclareMathOperator{\Aut}{Aut}
\newcommand{\WA}{\mr{WA}}

\begin{document}
\selectlanguage{english} 

\begin{abstract}
Following an approach of Ozawa, we show that several semidirect products are not weakly amenable. As a consequence, we are able to characterize the simply connected Lie groups that are weakly amenable.
\end{abstract}


\title{
\mbox{Weak amenability and simply connected Lie groups}
}
\thanks{The author is supported by ERC Advanced Grant no.~OAFPG 247321, the Danish National Research Foundation through the Centre for Symmetry and Deformation (DNRF92), and the Deutsche Forschungsgemeinschaft through the Collaborative Research Centre (SFB 878).}

\author{S{\o}ren Knudby}
\address{Mathematical Institute, University of M\"unster, \newline Einsteinstra\ss{}e 62, 48149 M\"unster, Germany.}
\email{knudby@uni-muenster.de}
\date{\today}
\maketitle

A locally compact group $G$ is weakly amenable if there is a net $(u_i)_{i\in I}$ of compactly supported Herz-Schur multipliers on $G$ converging to $1$ uniformly on compact subsets of $G$ and satisfying $\sup_i \|u_i\|_{B_2}\leq C$ for some $C\geq 1$ (see Section~\ref{sec:weak-amenability} for details). The infimum of those $C$ for which such a net exists is the \emph{weak amenability constant of $G$}, denoted here $\Lambda_\WA(G)$. Weak amenability was introduced by Cowling and Haagerup \cite{MR996553}.
By now, a lot is known about weak amenability, especially for (connected) Lie groups. Simple Lie groups are weakly amenable if and only if they have real rank at most one. The non-simple case was treated in \cite{MR2132866} in almost full generality (see Theorem~\ref{thm:finite-center} below).

A connected Lie group $G$ has a Levi decomposition $G = RS$ coming from a Levi decomposition of its Lie algebra $\mf g = \mf r\rtimes\mf s$. Here $\mf r$ is the solvable radical of $\mf g$ and $\mf s$ is a semisimple Lie algebra. The groups $R$ and $S$ are the connected Lie subgroups of $G$ associated with $\mf r$ and $\mf s$, respectively. The group $R$ is a closed normal solvable subgroup. The group $S$ is called a semisimple Levi factor of $G$ and is a semisimple Lie subgroup.
When $S$ has finite center, the authors of \cite{MR2132866} were able to completely characterize weak amenability of $G$.

\begin{thm}[\cite{MR2132866}]\label{thm:finite-center}
Let $G$ be a connected Lie group, and let $\mf g = \mf r\rtimes\mf s$ be a Levi decomposition of its Lie algebra. Let $S$ be the associated semisimple Levi factor and decompose the Lie algebra of $S$ into simple ideals as $\mf s = \mf s_1\oplus\cdots\oplus\mf s_n$. Suppose $S$ has finite center. The following are equivalent.
\begin{enumerate}
	\item $G$ is weakly amenable.
	\item For every $i = 1,\ldots,n$, one of the following holds:
	\begin{itemize}
		\item $\mf s_i$ has real rank zero;
		\item $\mf s_i$ has real rank one and $[\mf s_i,\mf r] = 0$.
	\end{itemize}
\end{enumerate}
In that case,
$$
\Lambda_\WA(G) = \prod_{i=1}^n\Lambda_\WA(S_i),
$$
where $S_i$ is the connected Lie subgroup of $G$ associated with $\mf s_i$.
\end{thm}

For any natural number $n\geq 1$, we let the group $\SL(2,\R)$ act on $\R^n$ by the unique irreducible representation of $\SL(2,\R)$ of dimension $n$. The group $\SL(2,\R)$ also acts on the Heisenberg group $H_{2n+1}$ of dimension $2n+1$ by fixing the center and acting on the vector space $\R^{2n}$ by the unique irreducible representation on $\R^{2n}$.

Apart from some structure theory, the proof of Theorem~\ref{thm:finite-center} relies on the following result whose proof occupies \cite{MR1245415} and the majority of \cite{MR2132866}.
\begin{thm}[\cite{MR2132866},\cite{MR1245415}]\label{thm:no-cover}
The following groups are not weakly amenable:
\begin{itemize}
	\item $\R^n\rtimes\SL(2,\R)$, where $n\geq 2$.
	\item $H_{2n+1}\rtimes\SL(2,\R)$, where $n\geq 1$.
\end{itemize}
\end{thm}

In this article, using a recent result of Ozawa \cite{MR2914879} about weakly amenable groups, we are able to give a new (and much simpler) proof of Theorem~\ref{thm:no-cover}. Ozawa already noted in \cite{MR2914879} that his result gave a new proof of the non-weak amenability of $\Z^2\rtimes\SL(2,\Z)$, which immediately implies non-weak amenability of $\R^2\rtimes\SL(2,\R)$.

In the study of weak amenability and related properties for Lie groups, the simply connected Lie groups are often more challenging to handle than for instance the simple Lie groups with finite center (see e.g. \cite{MR1418350,MR3047470,delaat2,HKdL-Tstar,MR1079871}). This is partly due to the fact that such groups are often not matrix groups and thus more difficult to describe explicitly. In this article, we completely settle the weak amenability question for simply connected Lie groups (Theorem~\ref{thm:simply-connected}).

We show that the universal covering groups of the groups appearing in Theorem~\ref{thm:no-cover} are not weakly amenable. Let $\tilde\SL(2,\R)$ denote the universal covering group of $\SL(2,\R)$. The group $\tilde\SL(2,\R)$ acts on $\R^n$ and $H_{2n+1}$ through the actions of $\SL(2,\R)$. We prove the following:

\begin{thm}\label{thm:cover}
The following groups are not weakly amenable:
\begin{itemize}
	\item $\R^n\rtimes\tilde\SL(2,\R)$, where $n\geq 2$.
	\item $H_{2n+1}\rtimes\tilde\SL(2,\R)$, where $n\geq 1$.
\end{itemize}
\end{thm}

As an application of Theorem~\ref{thm:cover}, we are able to characterize weak amenability for all simply connected Lie groups.

\begin{thm}\label{thm:simply-connected}
Let $G$ be a connected, simply connected Lie group, and let $\mf g = \mf r\rtimes\mf s$ be a Levi decomposition of its Lie algebra. Let $S$ be the associated semisimple Levi factor and decompose the Lie algebra of $S$ into simple ideals as $\mf s = \mf s_1\oplus\cdots\oplus\mf s_n$. The following are equivalent.
\begin{enumerate}
	\item $G$ is weakly amenable.
	\item For every $i = 1,\ldots,n$, one of the following holds:
	\begin{itemize}
		\item $\mf s_i$ has real rank zero;
		\item $\mf s_i$ has real rank one and $[\mf s_i,\mf r] = 0$.
	\end{itemize}
\end{enumerate}
In that case,
$$
\Lambda_\WA(G) = \prod_{i=1}^n\Lambda_\WA(S_i),
$$
where $S_i$ is the connected Lie subgroup of $G$ associated with $\mf s_i$.
\end{thm}

We expect that Theorem~\ref{thm:simply-connected} also holds without the assumption of simple connectedness, but we have not been able to establish this.

\section{Weak amenability and semidirect products}\label{sec:weak-amenability}
Let $G$ be a locally compact group. A complex, continuous function $u\colon G\to\C$ is a \emph{Herz-Schur multiplier} if there are a Hilbert space $\mc H$ and two bounded continuous functions $P,Q\colon G\to\mc H$ such that
$$
u(y^{-1}x) = \la P(x),Q(y)\ra  \quad\text{for every } x,y\in G.
$$
The Herz-Schur norm of $u$ is $\|u\|_{B_2} = \inf\{\|P\|_\infty\|Q\|_\infty\}$, where the infimum is taken over all $P,Q$ as above. There are other well-known descriptions of Herz-Schur multipliers \cite{MR753889}, \cite{MR1180643}, \cite[Theorem~5.1]{MR1818047}.

Recall that the group $G$ is weakly amenable if there is a net $(u_i)_{i\in I}$ of compactly supported Herz-Schur multipliers on $G$ converging to $1$ uniformly on compact subsets of $G$ and satisfying $\sup_i \|u_i\|_{B_2}\leq C$ for some $C\geq 1$. The infimum of those $C$ for which such a net exists is denoted $\Lambda_\WA(G)$, with the understanding that $\Lambda_\WA(G) = \infty$ if $G$ is not weakly amenable.
We refer to \cite[Section~12]{MR2391387} for a nice introduction to weak amenability. We list below the behaviour of the weak amenability constant under some relevant group constructions (see e.g. \cite[Section~1]{MR996553} and \cite{Jolissaint-proper}). These results will be needed in the proof of Theorem~\ref{thm:simply-connected}.

When $K$ is a compact normal subgroup of $G$,
\begin{align}\label{eq:mod-compact}
\Lambda_\WA(G/K) = \Lambda_\WA(G).
\end{align}

For a closed subgroup $H$ of $G$, 
\begin{align}\label{eq:subgroup}
\Lambda_\WA(H) \leq \Lambda_\WA(G),
\end{align}
and if $H$ is moreover co-amenable in $G$ (and $G$ is second countable), equality holds:
\begin{align}\label{eq:co-amenable}
\Lambda_\WA(H) = \Lambda_\WA(G).
\end{align}

For any two locally compact groups $G$ and $H$,
\begin{align}\label{eq:product}
\Lambda_\WA(G\times H)=\Lambda_\WA(G)\Lambda_\WA(H).
\end{align}

The following theorem is the basis for proving Theorems~\ref{thm:no-cover} and~\ref{thm:cover}. It relies on Ozawa's work \cite{MR2914879} using the technique in \cite[Corollary~2.12]{MR2680430} (see also \cite[Corollary~12.3.7]{MR2391387}). In \cite{MR2914879}, Ozawa proves that if a weakly amenable group $G$ has an amenable closed normal subgroup $N$, then there is a state on $L^\infty(N)$ which is both left $N$-invariant and conjugation $G$-invariant.

\begin{thm}\label{thm:ozawa}
Let $H\acts N$ be an action by automorphisms of a discrete group $H$ on a discrete group $N$, and let $G = N\rtimes H$ be the corresponding semidirect product group. Let $N_0$ be a proper subgroup of $N$. Suppose
\begin{enumerate}
	\item $H$ is not amenable;
	\item $N$ is amenable;
	\item $N_0$ is $H$-invariant;
	\item For every $x\in N\setminus N_0$, the stabilizer of $x$ in $H$ is amemable.
\end{enumerate}
Then $G$ is not weakly amenable.
\end{thm}
\begin{proof}
We suppose that $G$ is weakly amenable and derive a contradiction. By \cite[Theorem~A]{MR2914879}, there is an $N$-invariant mean $\mu$ on $\ell^\infty(N)$ which is moreover $H$-invariant, where $H$ acts on $N$ by conjugation.

Since $N_0$ is $H$-invariant, the action $H\acts N$ restricts to an action $H\acts N\setminus N_0$. Let $S$ be a system of representatives for the $H$-orbits in $N\setminus N_0$. For any $x\in S$, let
$$
H_x = \{ h\in H \mid h.x = x\}
$$
be the stabilizer subgroup of $x$ in $H$. We make the following identification of $H$-sets,
$$
N = N_0 \sqcup\bigsqcup_{x\in S} H/H_x.
$$

The stabilizer subgroup $H_x$ is amenable by assumption, so we may choose a left $H_x$-invariant mean $\mu_x$ on $\ell^\infty(H_x)$. Define $\phi_x\colon\ell^\infty(H)\to\ell^\infty(H/H_x)$ by averaging by $\mu_x$, that is,
$$
\phi_x(f)(h H_x) = \int_{H_x} f(h y) \ d\mu_x(y), \qquad f\in\ell^\infty(H).
$$
Then $\phi_x$ is unital, positive and $H$-equivariant. Collecting these maps, we obtain a unital, positive, $H$-equivariant map $\ell^\infty(H)\to\ell^\infty(N\setminus N_0)$. Since $H$ is not amenable, the $H$-invariant mean $\mu$ is concentrated on $N_0$. But this contradicts the fact that $\mu$ is also $N$-invariant.
\end{proof}

\section{Some semidirect product groups}\label{sec:semidirect}

For any natural number $n\geq 1$, the group $\SL(2,\R)$ has a unique irreducible representation on $\R^n$ (see \cite[p.~107]{MR0430163}). It is described explicitly in \cite[p.~710]{MR1245415}. The semidirect product $\R^n\rtimes\SL(2,\R)$ is defined using this representation. It is clear from the explicit description of the action in \cite[p.~710]{MR1245415} that $\SL(2,\Z)$ leaves the integer lattice $\Z^n$ invariant so that $\Z^n\rtimes\SL(2,\Z)$ is a well-defined subgroup of $\R^n\rtimes\SL(2,\R)$.

Let $H_{2n+1}$ denote the real Heisenberg group of dimension $2n+1$. We realize the Heisenberg group as $\R^{2n}\times\R$ with group multiplication given by
$$
(u_1,t_1)(u_2,t_2) = (u_1 + u_2, t_1 + t_2 + \la u_1, J u_2\ra)
$$
where $J$ is the symplectic $2n\times 2n$ matrix defined by
$$
J_{ij} = \left\{\begin{array}{ll}
	(-1)^j & \text{if } i+j = 2n+1,\\
	0 & \text{otherwise.}
\end{array}\right.
$$
For $j = 1,\ldots,2n$, let
$$
\alpha_j = \binom{2n-1}{j-1}^{1/2}.
$$
The irreducible representation $Z$ of $\SL(2,\R)$ of dimension $2n$ can be realized (in a different way than above) as
$$
Z(A)_{ij} = \sum_{l=0}^{2n} \binom{j-1}{l} \binom{2n-j}{2n-i-l} \alpha_i^{-1}\alpha_j a^{2n-i-l} b^{l} c^{i+l-j} d^{j-l-1}
$$
where
$$
A = \begin{pmatrix}
	a & b \\
	c & d
\end{pmatrix} \in \SL(2,\R).
$$
We refer to \cite[Section~2.1]{MR2132866} for more details. In \cite{MR2132866}, it is shown that the map $\bar Z\colon\SL(2,\R)\to\Aut(H_{2n+1})$ given by
$$
\bar Z(A)(u,t) = (Z(A)u,t), \qquad A\in\SL(2,\R),\ (u,t)\in H_{2n+1},
$$
defines an action by automorphisms of $\SL(2,\R)$ on $H_{2n+1}$. It is with respect to the action $\bar Z$ that we define the semidirect product $H_{2n+1}\rtimes\SL(2,\R)$.

Consider the lattice $\Lambda_{2n} = \alpha_1^{-1}\Z \oplus\cdots\oplus \alpha_{2n}^{-1}\Z$ in $\R^{2n}$ and let
\begin{align*}
\Gamma_{2n+1} = \left\{(u,t)\in H_{2n+1} \mid u\in\Lambda_{2n},\ t\in \frac{1}{N}\Z \right\},
\end{align*}
where $N = \alpha_1^2\cdots\alpha_{2n}^2$.
\begin{lem}
$\Gamma_{2n+1}$ is a discrete subgroup of $H_{2n+1}$ which is invariant under the action of $\SL(2,\Z)$.
\end{lem}
\begin{proof}
Observe that $\alpha_{2n+1-j} = \alpha_j$ for any $j=1,\ldots,2n$. It follows that $J\Lambda_{2n} = \Lambda_{2n}$, and $\la u_1,Ju_2\ra \in \frac1N \Z$ for any $u_1,u_2\in\Lambda_{2n}$. This shows that $\Gamma_{2n+1}$ is a subgroup of $H_{2n+1}$, and clearly $\Gamma_{2n+1}$ is discrete.
It is easily checked that if $A\in\SL(2,\Z)$, then $Z(A)\Lambda_{2n} \subseteq\Lambda_{2n}$. It follows that $\Gamma_{2n+1}$ is invariant under $\SL(2,\Z)$.
\end{proof}

Let $\tilde\SL(2,\R)$ be the universal covering group of $\SL(2,\R)$. The Lie group $\tilde\SL(2,\R)$ is simply connected with a covering homomorphism $\pi\colon\tilde\SL(2,\R)\to\SL(2,\R)$. The kernel of $\pi$ is a discrete normal subgroup of $\tilde\SL(2,\R)$ isomorphic to the group of integers. We let $\tilde\SL(2,\R)$ act on $\R^n$ and $H_{2n+1}$ through $\SL(2,\R)$, and in this way we obtain the semidirect products
$$
\R^n\rtimes\tilde\SL(2,\R) \quad\text{and}\quad H_{2n+1}\rtimes\tilde\SL(2,\R).
$$
We define the subgroup $\tilde\SL(2,\Z)$ of $\tilde\SL(2,\R)$ to be $\tilde\SL(2,\Z) = \pi^{-1}(\SL(2,\Z))$ and obtain the semidirect products
$$
\Z^n\rtimes\tilde\SL(2,\Z) \quad\text{and}\quad \Gamma_{2n+1}\rtimes\tilde\SL(2,\Z).
$$

\begin{lem}\label{lem:proper}
A proper, real algebraic subgroup of $\SL(2,\R)$ is amenable.
\end{lem}
\begin{proof}
Let $H$ be a proper, real algebraic subgroup of $\SL(2,\R)$. By a theorem of Whitney \cite[Theorem~3]{MR0095844}, $H$ has only finitely many components (in the usual Hausdorff topology) (see also \cite[Theorem~3.6]{MR1278263}). Hence, it suffices to show that the identity component $H^0$ of $H$ is amenable.

Since $H^0$ is a connected, proper, closed subgroup of $\SL(2,\R)$, its Lie algebra $\mf h$ is a proper Lie subalgebra of $\mf{sl}(2,\R)$. Hence, the dimension of $\mf h$ is at most two, and $\mf h$ must be solvable. So $H^0$ is solvable and, in particular, amenable.
\end{proof}

\begin{lem}
Let $n\geq 2$. For any $x\in\Z^n$ with $x\neq 0$, the stabilizer of $x$ in $\tilde\SL(2,\Z)$ is amenable.
\end{lem}
\begin{proof}
The stabilizer in $\tilde\SL(2,\Z)$ is precisely the preimage under $\pi$ of the stabilizer in $\SL(2,\Z)$. Since the kernel of $\pi$ is amenable, and amenability is preserved under extensions, it suffices to show that the stabilizer in $\SL(2,\Z)$ is amenable.

The stabilizer of $x$ in $\SL(2,\R)$ is a real algebraic subgroup. Moreover, since $x\neq 0$, the stabilizer of $x$ is proper, and hence by Lemma~\ref{lem:proper}, the stabilizer of $x$ in $\SL(2,\R)$ is amenable. It follows that the stabilizer in the closed subgroup $\SL(2,\Z)$ is amenable.
\end{proof}

In the following lemma, we consider the action of $\tilde\SL(2,\Z)$ on $\Gamma_{2n+1}$ previously described. Note that the center of $\Gamma_{2n+1}$ is precisely $\{(u,t)\in \Gamma_{2n+1} \mid u = 0\}$.

\begin{lem}
Let $n\geq 1$. For any non-central $x\in \Gamma_{2n+1}$, the stabilizer of $x$ in $\tilde\SL(2,\Z)$ is amenable.
\end{lem}
\begin{proof}
As before, it suffices to prove that the stabilizer of $x$ in $\SL(2,\R)$ is amenable. If we write $x = (u,t)\in\Gamma_{2n+1}$, then the stabilizer of $x$ in $\SL(2,\R)$ is
$$
\{A\in\SL(2,\R) \mid Z(A) u = u \}
$$
Clearly, this is a real algebraic subgroup of $\SL(2,\R)$. Moreover, since $u\neq 0$, the stabilizer of $x$ is proper. By Lemma~\ref{lem:proper}, the stabilizer of $x$ in $\SL(2,\R)$ is amenable.
\end{proof}

\begin{proof}[Proof of Theorem~\ref{thm:cover}]
Case of $\R^n\rtimes\tilde\SL(2,\R)$: The group $\Z^n\rtimes\tilde\SL(2,\Z)$ is a closed subgroup of $\R^n\rtimes\tilde\SL(2,\R)$ (in fact a lattice), so it suffices to prove that $\Z^n\rtimes\tilde\SL(2,\Z)$ is not weakly amenable. This is a direct application of Theorem~\ref{thm:ozawa} with $H = \tilde\SL(2,\Z)$, $N = \Z^2$, and $N_0 = \{0\}$.

Case of $H_{2n+1}\rtimes\tilde\SL(2,\R)$: The group $\Gamma_{2n+1}\rtimes\tilde\SL(2,\Z)$ is a closed subgroup of $H_{2n+1}\rtimes\tilde\SL(2,\R)$ (in fact a lattice), so it suffices to prove that $\Gamma_{2n+1}\rtimes\tilde\SL(2,\Z)$ is not weakly amenable. This is a direct application of Theorem~\ref{thm:ozawa} with $H = \tilde\SL(2,\Z)$, $N = \Gamma_{2n+1}$, and $N_0$ equal to the center of $\Gamma_{2n+1}$.
\end{proof}

\begin{proof}[Proof of Theorem~\ref{thm:no-cover}]
Similar to the proof of Theorem~\ref{thm:cover}. One just has to replace $\tilde\SL(2,\Z)$ by $\SL(2,\Z)$.
\end{proof}

\begin{rem}
Note that we have in fact proved that the following discrete groups are not weakly amenable:
\begin{itemize}
	\item $\Z^n\rtimes\SL(2,\Z)$, where $n\geq 2$.
	\item $\Z^n\rtimes\tilde\SL(2,\Z)$, where $n\geq 2$.
	\item $\Gamma_{2n+1}\rtimes\SL(2,\Z)$, where $n\geq 1$.
	\item $\Gamma_{2n+1}\rtimes\tilde\SL(2,\Z)$, where $n\geq 1$.
\end{itemize}
\end{rem}

\section{Simply connected Lie groups}
This section contains the proof of Theorem~\ref{thm:simply-connected}. First we review the structure theory of Lie groups that is needed in the proof, in particular the Levi decomposition (see \cite[Theorem~3.18.13]{MR746308}).

Let $G$ be a connected Lie group with Lie algebra $\mf g$. We denote solvable radical of $\mf g$ by $\mr{rad}(\mf g)$ or $\mf r$. In other words, $\mf r$ the maximal solvable ideal of $\mf g$. There is a semisimple Lie subalgebra $\mf s$ of $\mf g$ such that $\mf g = \mf r\rtimes\mf s$. The semisimple Lie algebra $\mf s$ is a direct sum $\mf s = \mf s_1\oplus\cdots\oplus\mf s_n$ of simple Lie algebras (for some $n\geq 0$). If $R$ and $S$ denote the connected Lie subgroups of $G$ associated with $\mf r$ and $\mf s$, respectively, then $R$ is a closed, normal subgroup of $G$ and $S$ is maximal semisimple but not necessarily closed. Moreover, $G = RS$ as a set. The group $S$, which in general is not unique, is called a semisimple Levi factor. If $G$ is simply connected, then $S$ is closed, $R\cap S = \{1\}$ and $G = R\rtimes S$ as Lie groups.

For a connected, simply connected Lie group $G$, we will prove that the following are equivalent.
\begin{enumerate}
	\item $G$ is weakly amenable.
	\item For every $i = 1,\ldots,n$, one of the following holds:
	\begin{itemize}
		\item $\mf s_i$ has real rank zero;
		\item $\mf s_i$ has real rank one and $[\mf s_i,\mf r] = 0$.
	\end{itemize}
\end{enumerate}

The following proposition can be found in \cite{cornulier-jlt} (see the proof of \cite[Proposition~1.9]{cornulier-jlt}) and essentially appears already in \cite{MR2132866}. Let $\mf v_{n+1}\rtimes\mf{sl}_2$ denote the Lie algebra of $\R^{n+1}\rtimes\SL(2,\R)$ and let $\mf h_{2n+1}\rtimes\mf{sl}_2$ denote the Lie algebra of $H_{2n+1}\rtimes\SL(2,\R)$.
\begin{prop}[\cite{MR2132866},\cite{cornulier-jlt}]\label{prop:dichotomy}
Let $\mf g$ be a Lie algebra with solvable radical $\mf r$ and a Levi decomposition $\mf g = \mf r\rtimes\mf s$. Write $\mf s = \mf s_c\oplus\mf s_{nc}$ by separating compact factors $\mf s_c$ (rank zero) and non-compact factors $\mf s_{nc}$ (positive rank). Exactly one of the following holds.
\begin{itemize}
	\item[(a)] All non-compact factors of $\mf s$ commute with $\mf r$: $[\mf r,\mf s_{nc}] = 0$.
	\item[(b)] $\mf g$ has a subalgebra $\mf h$ isomorphic to $\mf v_{n+1}\rtimes\mf{sl}_2$ or $\mf h_{2n+1}\rtimes\mf{sl}_2$ for some $n\geq 1$, where $\mr{rad}(\mf h)\subseteq\mf r$ and $\mf{sl}_2\subseteq\mf s_{nc}$.
\end{itemize}
\end{prop}

\begin{lem}\label{lem:levi-quotient}
Let $G$ be $\R^{n+1}\rtimes\tilde\SL(2,\R)$ or $H_{2n+1}\rtimes\tilde\SL(2,\R)$, where $n\geq 1$. The semisimple Levi factor of $G$ is unique, and if $Z$ is a central subgroup of $G$ contained in the semisimple Levi factor, then $G/Z$ is not weakly amenable.
\end{lem}
\begin{proof}
If $R$ is the solvable radical of $G$, then $[R,R]$ is central in $G$: the commutator group $[R,R]$ is trivial in the first case and in the second case equal to the center of $H_{2n+1}$, which is also central in $H_{2n+1}\rtimes\tilde\SL(2,\R)$. By \cite[Theorem~3.18.13]{MR746308}, any two Levi factors of $G$ are conjugate by an element of $[R,R]$, and hence, in our case, they are actually equal.

The center of $\tilde\SL(2,\R)$ is isomorphic to the group of integers. If $Z$ is the trivial group, we are done by Theorem~\ref{thm:cover}. Otherwise, $Z$ has finite index in the center of $\tilde\SL(2,\R)$, and $G/Z$ is isomorphic up to a finite covering to $\R^{n+1}\rtimes\SL(2,\R)$ or $H_{2n+1}\rtimes\SL(2,\R)$. Then we are done by Theorem~\ref{thm:no-cover} and equation \eqref{eq:mod-compact}.
\end{proof}

\begin{proof}[Proof of Theorem~\ref{thm:simply-connected}]
When $G$ is simply connected, the Levi decomposition expresses $G$ as a semidirect product $G = R\rtimes S$, where $R$ is the solvable radical and $S$ is closed and semisimple (see \cite[Theorem~3.18.13]{MR746308}). Both $R$ and $S$ are simply connected. Decompose the Lie algebra of $S$ into simple ideals $\mf s = \mf s_1\oplus\cdots\oplus\mf s_n$. Recall that two simply connected Lie groups with isomorphic Lie algebras are isomorphic. If $S_i$ is a simply connected Lie group with Lie algebra $\mf s_i$, then $S$ is isomorphic to the direct product $S_1\times\cdots\times S_n$. We split $S$ into the compact factors $S_c$ and non-compact factors $S_{nc}$, $S = S_c \times S_{nc}$.

Assume first (2) holds. Then $S_{nc}$ is a product of simple factors of rank one, so $S_{nc}$ is weakly amenable (see \cite{MR996553}, \cite{MR1079871}). Morevoer, $S_{nc}$ is a direct factor in $G$ and the quotient $G/S_{nc}$ is $R\rtimes S_c$. As $S_c$ is compact and $R$ is solvable, the group $G/S_{nc} = R\rtimes S_c$ is amenable. It follows from \eqref{eq:co-amenable} and \eqref{eq:product} that $G$ is weakly amenable with
$$
\Lambda_\WA(G) = \Lambda_\WA(S_{nc}) = \prod_{i=1}^n \Lambda_\WA(S_i).
$$
For the last equality, we also used the obvious fact that $\Lambda_\WA(S_c) = 1$, since $S_c$ is compact.

Assume next that (2) does not hold. Let $\mf v_{k+1}\rtimes\mf{sl}_2$ denote the Lie algebra of $\R^{k+1}\rtimes\SL(2,\R)$, and let $\mf h_{2k+1}\rtimes\mf{sl}_2$ denote the Lie algebra of $H_{2k+1}\rtimes\SL(2,\R)$.

If some $\mf s_i$ has real rank at least two, then the simple Lie group $S_i$ is not weakly amenable (see \cite[Theorem~1]{MR1418350}), and since $S_i$ is closed in $G$, it follows that $G$ is not weakly amenable. Otherwise some $\mf s_i$ has real rank one, but $[\mf s_i,\mf r]\neq 0$. By Proposition~\ref{prop:dichotomy}, the Lie algebra $\mf g$ contains a subalgebra $\mf h$ isomorphic to $\mf v_{k+1}\rtimes\mf{sl}_2$ or $\mf h_{2k+1}\rtimes\mf{sl}_2$ for some $k\geq 1$, where $\mr{rad}(\mf h)\subseteq\mf r$ and $\mf{sl}_2\subseteq\mf s$.
Hence $G$ contains a Lie subgroup $H$ locally isomorphic to $\R^{k+1}\rtimes\SL(2,\R)$ or $H_{2k+1}\rtimes\SL(2,\R)$. We claim that $H$ is closed and not weakly amenable.

Let $\mf h = \mf r_0\rtimes\mf s_0$ be a Levi decomposition of $\mf h$, that is, $\mf r_0$ is $\mf v_{k+1}$ or $\mf h_{2k+1}$ and $\mf s_0 = \mf{sl}_2$. Let $R_0$ and $S_0$ denote the connected Lie subgroups of $G$ associated with $\mf r_0$ and $\mf s_0$, respectively.

The group $S_0$ is a semisimple connected Lie subgroup of $S$ and hence closed \cite[p.~615]{MR0048464}. Moreover, $S_0$ is locally isomorphic to $\SL(2,\R)$. The group $R_0$ is simply connected and closed in $R$ (see \cite[Theorem~3.18.12]{MR746308}). Clearly, $S_0$ normalizes $R_0$ and $H=R_0S_0$, and since moreover $R\cap S =\{1\}$, we get that $H = R_0\rtimes_\beta S_0$, where $\beta$ denotes the conjugation action of $S_0$ on $R_0$. In particular, $H$ is closed in $G$.

Let $\tilde S_0$ be the universal cover of $S_0$ (so $\tilde S_0 = \tilde\SL(2,\R)$) and consider the semidirect product $\tilde H = R_0\rtimes_\beta\tilde S_0$, where $\tilde S_0$ acts on $R_0$ through the covering $\tilde S_0\to S_0$ and the action of $S_0$ on $R_0$. The group $\tilde H$ is simply connected and hence isomorphic to $\R^{k+1}\rtimes\tilde\SL(2,\R)$ or $H_{2k+1}\rtimes\tilde\SL(2,\R)$. The group $H$ is a quotient of $\tilde H$ by a central subgroup contained in the Levi factor of $\tilde H$, so by Lemma~\ref{lem:levi-quotient} the group $H$ is not weakly amenable. It follows that $G$ is not weakly amenable.
\end{proof}


\end{document}